\documentclass[12pt]{article}

\usepackage[12pt]{extsizes}
\usepackage[T2A,T1]{fontenc}
\usepackage[utf8]{inputenc}
\usepackage[main=english, russian]{babel}

\usepackage{ccaption}
\captiondelim{. }
\usepackage{csquotes}
\usepackage{stmaryrd}
\usepackage{stackengine}

\usepackage{mathrsfs}
\usepackage[dvipsnames]{xcolor}

\usepackage{amsfonts}
\usepackage{amsmath}
\usepackage{amssymb}
\usepackage{amsthm}
\usepackage{asymptote}
\usepackage{dsfont}
\usepackage{enumitem}
\usepackage{epigraph}
\usepackage{fancyhdr} % Headers and footers
\usepackage{graphicx}
\usepackage[a4paper, left=20mm, right=20mm, top=20mm, bottom=20mm, bindingoffset=0cm]{geometry}
\usepackage{indentfirst}
\usepackage{tikz}
\usepackage{tabu}
\usepackage{diagbox}
\usepackage{hyperref}
\usepackage{mathtools}
\usepackage{multicol}
\usepackage{ucs}
\usepackage{lipsum}

\fontencoding{T1}\selectfont

\usepackage[framemethod=TikZ]{mdframed}

\newcommand{\definebox}[3]{%
    \newcounter{#1}
    \newenvironment{#1}[1][]{%
        \stepcounter{#1}%
        \mdfsetup{%
            frametitle={%
            \tikz[baseline=(current bounding box.east),outer sep=0pt]
            \node[anchor=east,rectangle,fill=white]
            {\strut #2~\csname the#1\endcsname\ifstrempty{##1}{}{##1}};}}%
        \mdfsetup{innertopmargin=1pt,linecolor=#3,%
            linewidth=3pt,topline=true,
            frametitleaboveskip=\dimexpr-\ht\strutbox\relax,}%
        \begin{mdframed}[]
            \relax%
            }{
        \end{mdframed}}%
}

\definebox{theorem_boxed}{\iflanguage{english}{Theorem}{\foreignlanguage{russian}{Теорема}}}{ForestGreen!24}
\definebox{definition_boxed}{\iflanguage{english}{Definition}{\foreignlanguage{russian}{Определение}}}{blue!24}
\definebox{task_boxed}{\iflanguage{english}{Problem}{\foreignlanguage{russian}{Задача}}}{orange!24}
\definebox{paradox_boxed}{\iflanguage{english}{Paradox}{\foreignlanguage{russian}{Парадокс}}}{red!24}

\theoremstyle{plain}

\iflanguage{english}{
  \newtheorem{theorem}{Theorem}
  \newtheorem{lemma}{Lemma}
  \newtheorem{corollary}{Corollary}
  \newtheorem{definition}{Definition}
  \newtheorem{exercise}{Exercise}
  \newtheorem{hypotheses}{Hypothesis}
  \newtheorem{statement}{Statement}
  \newtheorem{consiquence}{Consiquence}
}{
  \newtheorem{theorem}{Теорема}
  \newtheorem{lemma}{Лемма}

  \newtheorem{statement}{Утверждение}
  
}

\theoremstyle{remark}

\iflanguage{english}{
  \newtheorem{remark}{Remark}
  \newtheorem{example}{Example}
}{
  \newtheorem{remark}{Замечание}
  
}

\DeclarePairedDelimiter\abs{\lvert}{\rvert}
\DeclarePairedDelimiter\floor{\lfloor}{\rfloor}

\DeclarePairedDelimiter\lr{(}{)}
\DeclarePairedDelimiter\set{\{}{\}}

\newcommand{\N}{\mathbb{N}}
\newcommand{\R}{\mathbb{R}}
\newcommand{\E}{\mathsf{E}}

\newcommand{\D}{\mathsf{D}}
\newcommand{\M}{\mathcal{M}}

\DeclareGraphicsExtensions{.pdf,.png,.jpg}

\begin{document}

\title{On the Properties of the Maximum of a Random Assignment Process\thanks{This work was supported by the Ministry of Education and Science of the Russian Federation, grant no.~075-15-2025-013.}}
\author{T.\,D.~Moskalenko\footnote{Saint Petersburg State University, Saint Petersburg, Russia; e-mail: \texttt{teamofey1982@gmail.com}}}
\maketitle

\begin{abstract}
We study the maximum of the random assignment process on rectangular matrices. We derive first-order asymptotics for the expected maximum, prove a law of large numbers under mild tail assumptions, and obtain exponential upper bounds for the probabilities of large deviations.\\

\end{abstract}

\section{Introduction}\label{sec:intro}
% !TEX root = main_en.tex
% English version of: 1 - Введение.tex
% Preserves math, labels, and citation keys.

Let $[n]$ denote the set $\{1,\dots,n\}$ of natural numbers.

A \emph{generalized permutation} is an injective mapping $\pi\colon [n]\to[m]$ with $m\ge n$.

Let
\(
\bigl(X_{ij}\bigr)_{
  \substack{1\leqslant i\le n\\ 1\leqslant j\le m}
}
\)
be an $n\times m$ array with $m\ge n$.

The \emph{assignment process} is the family of random variables
\(
\Bigl\{\,S(\pi)\coloneqq \sum_{k=1}^{n} X_{k,\pi(k)}\,\Bigr\},
\)
where $\pi\colon [n]\to[m]$ ranges over all generalized permutations.
We consider the assignment process for a random matrix in which the entries $X_{ij}$ are independent and identically distributed (i.i.d.) random variables. For notational convenience, we will use a random variable $X$ with the same distribution as the matrix entries, and finite sequences of independent random variables $X_1,\dots,X_n$ with the same law.

The assignment problem consists in studying the maximal and minimal values of the assignment process
\[
    \M_{n, m} \coloneqq \max\limits_{\pi} S(\pi), \quad \M_{n, m}^{-} \coloneqq \min\limits_{\pi} S(\pi).
\]

Here the quantities $X_{ij}$ are interpreted as costs, rewards, or weights, depending on the context of the problem. In the minimization problem they are treated as costs, whereas in the maximization problem they are interpreted as profits or quality scores.

The assignment problem with random costs has a rich history and numerous applications across mathematics and engineering: from modeling bipartite graphs in practical industrial settings to the design of low-complexity algorithms.

This formulation arises in auction problems, in resource allocation for wireless communication systems, and in charging-control problems for electric vehicles. We mention only a few modern examples. In wireless communication models where a set of users transmits over a common wireless channel, one must allocate channel resources among users. In the work of Bai et al.~\cite{BWZ10}, resource allocation for OFDMA (Orthogonal Frequency Division Multiple Access) is formulated as maximizing an assignment process with random pairwise link throughputs. In a subsequent paper, Bai et al.~\cite{BWLZ13} studied the outage probability in subchannel allocation, where an outage (i.e., a failed transmission) occurs when a user receives insufficient channel resources. A related problem was analyzed by Chen et al.~\cite{YYHL21}, who considered the joint design of resource-block allocation and the assignment of modulation-and-coding schemes.

Beyond wireless systems, the assignment problem is a natural model for organizing smart transportation via driver–passenger matching; see Ke et al.~\cite{KXYY20}, where a two-stage optimization scheme is also proposed.

For completeness, we recall a few classical theoretical results.

In~\cite{MP87}, Mezard and Parisi put forward the remarkable conjecture that if $X$ has the standard exponential distribution, then
\begin{equation}
    \lim\limits_{n\to \infty} \E \M_{n, n}^- = \frac{\pi^2}{6}.
    \label{eq:mp_h}
\end{equation}

That conjecture was later proved by Aldous~\cite{Al01}; moreover, it turned out that the result holds for all distributions of nonnegative random variables that have the same unit density at zero as the standard exponential distribution.
For the exponential distribution an exact formula was even obtained~\cite{NPS05}:
\[
    \E \M_{n, n}^- = \sum\limits_{k = 1}^{n} k^{-2}.
\]

Note that these results concern the minimum of variables whose distributions are bounded below.
For variables with unbounded distributions the situation is completely different.
A finite limit as in~\eqref{eq:mp_h} no longer exists, and one must examine the asymptotic behavior of extreme values.
In this case it is natural to consider the maximum rather than the minimum.
Thus, in~\cite{MS21} Mordant and Segers showed that if $X$ has the standard Gaussian distribution, then
\begin{equation}
    \E \M_{n,n} = n\cdot \sqrt{2 \log\log n} \cdot (1 + o(1)), \quad n\to \infty.
    \label{eq:gaussian-expectation}
\end{equation}

This result was generalized by Lifshits and Tadevosyan~\cite{LT22} to a broad class of random variables whose distributions have rapidly decaying and sufficiently regular tails at infinity.
The result below~\cite{LT22} plays a central role in the present paper.
We recall the following classical definition.

A function $L(\cdot)$ is said to be \textit{slowly varying} at zero if for any $x>0$
\[
    \lim\limits_{t \to 0} \frac{L(tx)}{L(t)} = 1.
\]

Define the \textit{upper quantile function} of the distribution of a random variable $X$ by
\[
    g(p) \coloneqq \inf \{r \colon \Prob(X \geqslant r) < p \}.
\]

\begin{theorem}[\cite{LT22}]
  \label{MT-expect-asy}
  Let $(X_{ij})$ be a random matrix of size $n\times n$ with i.i.d.\ entries.
  Assume that
  \begin{enumerate}
    \item $\E X^{-}<\infty$;
    \item $g(\cdot)$ is slowly varying at $0$ and $\lim_{p\to 0+} g(p)=+\infty$.
  \end{enumerate}
  Then
  \begin{equation}
    \E \M_{n,n} \;=\; n\, g\!\left(\frac{1}{n}\right)\bigl(1+o(1)\bigr),\qquad n\to\infty.
    \label{eq:general-expectation}
  \end{equation}
\end{theorem}

Formula~\eqref{eq:gaussian-expectation} is a special case of~\eqref{eq:general-expectation}.

The present paper continues this line of research around Theorem~\ref{MT-expect-asy}.
We extend \eqref{eq:general-expectation} to rectangular matrices (see Theorem~\ref{th:expect-asy}), establish a corresponding law of large numbers (Theorem~\ref{th:lln}), and obtain exponential bounds for large‑deviation probabilities (Theorem~\ref{th:exponential-bound-exp-tail}).

Rectangular matrices were also studied earlier in~\cite{NPS05,AS02,BCR02,CS99,LW02,W09}.

\section{Asymptotic behaviour of the expectation}\label{sec:expect}
The next theorem generalizes Theorem~\ref{MT-expect-asy} to the case of rectangular matrices.

\begin{theorem}
    \label{th:expect-asy}
    Consider a sequence of random matrices $(X_{ij})$ of size $n\times m(n)$ with $m\geqslant n$, whose entries are i.i.d.\ random variables.
    Assume that
    \begin{enumerate}
        \item $\E X^{-} < \infty$,
        \item $g(\cdot)$ is slowly varying at $0$ and $\lim\limits_{p\to 0+} g(p) = +\infty$.
    \end{enumerate}

    Then
    \[
        \E\M_{n, m} = n\cdot g\!\left(\frac{1}{m}\right)(1 + o(1)), \qquad n\rightarrow \infty.
    \]
\end{theorem}

Denote by $M_n \coloneqq \max\limits_{1 \leqslant k \leqslant n} X_k$ the maximum of a sample of size $n$ of i.i.d.\ copies of $X$.
Our estimates rely on the following auxiliary result from~\cite{LT22}.

\begin{lemma}[\cite{LT22}]
    \label{lm:max-expect}
    Let $\{X_{k}\}$ be i.i.d.
    Assume that
    \begin{enumerate}
        \item $\E X^{-} < \infty$,
        \item $g(\cdot)$ is slowly varying at $0$ and $\lim\limits_{p\to 0+} g(p) = +\infty$.
    \end{enumerate}

    Then
    \[
        \E M_{n} = g\left(\frac{1}{n}\right)(1 + o(1)), \quad n\rightarrow \infty.
    \]
\end{lemma}

\begin{proof}[Proof of Theorem~\ref{th:expect-asy}]
\emph{Upper bound.}
For any generalized permutation $\pi$ the sum $\sum_{k=1}^n X_{k,\pi(k)}$ does not exceed the sum of the rowwise maxima.
Therefore, by Lemma~\ref{lm:max-expect},
\[\E\M_{n,m} \leqslant \sum_{k = 1}^{n}\E M_{m} = n \cdot g\left(\frac{1}{m}\right)(1 + o(1)).\]

\emph{Lower bound.}
We use the greedy construction of a generalized permutation as in~\cite{LT22}, with a minor modification.
In the first row pick the maximum — it is distributed as $M_m$.
Restrict the problem to the \emph{accessible} part of the table: all rows except the first one and all columns except the one where the maximum was found.
At the $k$-th step, among the remaining $(n-k)$ rows we consider the first one.
In that row there are $(m-k)$ accessible entries, and the greedy algorithm picks their maximum.
After that we remove from consideration the column and the row where this entry was found, and so on.

At each step, in the ``accessible'' subtable the entries remain independent and identically distributed, with the same law as $X$.
As a result the algorithm produces a generalized permutation whose value does not exceed the maximum.
Hence
\[
    \E\M_n \geqslant \sum_{k = m - n + 1}^{m}\E M_k^{(k)}.
\]
Here and below the superscript $(k)$ indicates that $M_k^{(k)}$ is the maximum in an \emph{independent} sample of size $k$ with the same distribution as $X$.

Furthermore, for any $k$ one has $M_k^{(k)} \geqslant -X_1^{-}$ almost surely, and therefore $\E M_k^{(k)} \geqslant - \E X_1^-$.
Since $g$ is slowly varying at $0$, for any $\delta > 0$ we obtain
\begin{equation*}
    \begin{aligned}[b]
    \sum_{k = m - n + 1}^{m}\E M_k^{(k)}
    &=
    \sum_{k = m - n + 1}^{m - n + \floor{\delta n}}\E M_k^{(k)} +
    \sum_{k = m - n + \floor{\delta n} + 1}^{m}\E M_k^{(k)} \\
    &\geqslant
    -\sum_{k = 1}^{\floor{\delta n}}\E X_1^-
    +\sum_{k = m - n + \floor{\delta n} + 1}^{m}
    g\lr*{\frac{1}{k}}(1 + o(1)) \\
    &= O(n) + (1-\delta) n\cdot g\left(\frac{1}{m}\right)(1 + o(1)).
    \end{aligned}
\end{equation*}

Since $g(0+) = +\infty$, the first term is asymptotically negligible compared to the second one.
Letting $\delta \to 0$ yields
\begin{equation}
    \E\M_{n, m} \geqslant \sum_{k = m - n + 1}^{m}\E M_k^{(k)} \geqslant n\cdot g\left(\frac{1}{m}\right)(1 + o(1)).
    \label{eq:sum-expect-lower-bound}
\end{equation}
\end{proof}

\begin{remark}
    The case $m < n$ can be obtained by transposing the matrix.
    The asymptotics in that case is
    \[
    \E\M_{n,m} = m \cdot g\lr*{\frac{1}{n}}(1 + o(1)), \quad n \rightarrow \infty.
    \]
\end{remark}

\section{Law of Large Numbers}\label{sec:lbn}
\begin{theorem}
    \label{th:lln}
    Let $(X_{ij})$ be a random matrix of size $n\times m$ with i.i.d.\ entries, and for every $n$ assume $m\ge n$.
    Suppose that
    \begin{enumerate}
        \item $\E X^{-} < \infty$,
        \item $g(\cdot)$ is slowly varying at $0$ and $\lim\limits_{p\to 0+} g(p) = +\infty$.
    \end{enumerate}
    Then the following convergence in probability holds:
    \[
        \frac{\M_{n,m}}{n g\lr*{\frac{1}{m}}}   \xrightarrow{\Prob} 1, \quad n \to \infty.
    \]
\end{theorem}

\begin{lemma}
    \label{lm:upper-bound}
    Assume that
    \begin{enumerate}
        \item $\E X^{-} < \infty$,
        \item $g(\cdot)$ is slowly varying at $0$ and $\lim\limits_{p\to 0+} g(p) = +\infty$.
    \end{enumerate}
    Then for any $q\in\N$ and $s>0$, for all sufficiently large $n\in\N$, and for $\alpha>2^q$ one has
    \[
        \Prob\lr*{\left\lvert M_n \right\rvert^q \geqslant \alpha g\lr*{\frac{1}{n}}^q} \leqslant  \lr*{\frac{\alpha^{1/q}}{2}}^{-s} +  \lr*{\frac{\E X^-}{\alpha^{1/q}}}^s =\colon C_1\alpha^\frac{-s}{q}.
    \]
\end{lemma}

The function $g(\cdot)$ varies more slowly than any power function. We will use the following \emph{Potter bound}~\cite[Theorem~1.5.6]{BGT}.

\begin{statement}[Potter bound]
    Let $g(\cdot)$ be slowly varying at $0$.
    Then for any $\delta>0$ and $A>1$ there exists $x_0>0$ such that for all $0<x,y<x_0$,
    \[
        \frac{g\lr*{x}}{g\lr*{y}} \leqslant A \cdot \max\set*{
            \lr*{\frac{x}{y}}^{\delta},\lr*{\frac{y}{x}}^{\delta} }.
    \]
\end{statement}

\begin{proof}[Proof of Lemma~\ref{lm:upper-bound}]
    Apply the Potter bound with $A=2$, $\delta=\frac{1}{s}$, $x=\frac{2^{s}}{n\,\alpha^{s/q}}$, $y=\frac{1}{n}$ for large enough $n$ and $\alpha>2^{q}$.
    Then
    \begin{gather*}
        \frac{g\lr*{ \frac{2^s}{n\alpha^{s/q}} } }{g\left(\frac{1}{n}\right) } \leqslant 2\cdot \max\set*{\frac{\alpha^{1/q}}{2}, \frac{2}{\alpha^{1/q}}} = \alpha^{1/q}.
    \end{gather*}

    By the definition of $g$,

    \begin{equation}
        \label{eq:first-bound}
        \Prob\lr*{X \geqslant \alpha^{1/q} \cdot g\left(\frac{1}{n}\right) } \leqslant
        \Prob\lr*{X \geqslant g\lr*{ \frac{2^s}{n\alpha^{s/q}} }} =
        \frac{\lr*{\frac{\alpha^{1/q}}{2}}^{-s}} {n}.
    \end{equation}

    Since $g\!\lr*{\tfrac{1}{n}}>1$ for $n$ large enough, we have
    \begin{equation}
        \label{eq:second-bound}
        \Prob\lr*{ X \leqslant -\alpha^{1/q}  g\lr*{\frac{1}{n}}}
        = \Prob\lr*{ X^{-} \geqslant \alpha^{1/q}  g\lr*{\frac{1}{n}}}
        \leqslant \Prob\lr*{  X^{-} \geqslant \alpha^{1/q} }
        \leqslant  \frac{\E  X^{-}}{\alpha^{1/q}}.
    \end{equation}

    Using \eqref{eq:first-bound}–\eqref{eq:second-bound} and additionally assuming $n\ge s$, we estimate the maximum:
    \begin{align*}
        \Prob \left(\left\lvert M_n \right\rvert^q \geqslant \alpha g\left(\frac{1}{n}\right)^q \right)
        &= \Prob \left( \left\lvert M_n \right\rvert \geqslant \alpha^{1/q} g\left( \frac{1}{n} \right) \right) \\
        &= \Prob \left( M_n \geqslant \alpha^{1/q} g\left( \frac{1}{n} \right) \right)
         + \Prob \left( M_n \leqslant -\alpha^{1/q} g\left( \frac{1}{n} \right) \right) \\
        &\leqslant n \cdot \Prob\left( X \geqslant \alpha^{1/q} g\left( \frac{1}{n} \right) \right)
         + \Prob\left( X \leqslant -\alpha^{1/q} g\left( \frac{1}{n} \right) \right)^n \\
        &\leqslant \left( \frac{\alpha^{1/q}}{2} \right)^{-s}
         + \Prob\left( X \leqslant -\alpha^{1/q} g\left( \frac{1}{n} \right) \right)^s \\
        &\leqslant \left( \frac{\alpha^{1/q}}{2} \right)^{-s}
         + \left( \frac{\E X^-}{\alpha^{1/q}} \right)^s = C_1 \alpha^{-s/q}.
    \end{align*}
\end{proof}

\begin{proof}[Proof of Theorem~\ref{th:lln}]
    Apply Lemma~\ref{lm:upper-bound} with $q=2$ and $s=6$. Then for $\alpha>2^{q}=4$ and for all sufficiently large $r\in\N$ we have
    
    \begin{equation}
        \begin{aligned}[b]
            \frac{\D\lr*{M_{r}}}{g\lr*{\frac{1}{{r}}}^2}
            &\leqslant
            \frac{\E\lr*{M_{r}^2}}{g\lr*{\frac{1}{{r}}}^2}  \\
            &= \int\limits_{0}^{\infty}\Prob \lr*{M_{r}^2 \geqslant \alpha g\lr*{\frac{1}{{r}}}^2}\mathrm{d}\alpha
            \\
            &\leqslant 4 + \int\limits_{4}^{\infty}\Prob \lr*{M_{r}^2 \geqslant \alpha g\lr*{\frac{1}{{r}}}^2} \mathrm{d}\alpha
            \\
            &\leqslant 4 + \int\limits_{4}^{\infty} C_1\alpha^{-3} \mathrm{d}\alpha \eqqcolon C_2.
        \end{aligned}
        \label{eq:var-bound}
    \end{equation}

    We estimate $\M_{n,m}$ from above and below as in Theorem~\ref{th:expect-asy}:
    \begin{equation}
        \label{eq:two-way-bound}
        \sum\limits_{k=m-n+1}^{m} M_{k}^{(k)} \;\le\; \M_{n,m} \;\le\; \sum\limits_{k=1}^{n} M_{m}^{(k)},
    \end{equation}
    where the summands in both sums are independent.

    For the upper bound, apply Chebyshev to $\frac{\sum\limits_{k=1}^{n} M_{m}^{(k)}}{n\, g\!\lr*{\tfrac{1}{m}}}$.
    Using \eqref{eq:var-bound}, for any fixed $t$ and all sufficiently large $m$,
    \[
        \Prob\set*{
            \abs*{
                \frac{\sum\limits_{k=1}^{n}M_{m}^{(k)}}{n g\lr*{\frac{1}{{m}}}} -
                \frac{\sum\limits_{k=1}^{n}\E M_{m}}{n g\lr*{\frac{1}{{m}}}}
            } \geqslant t
        }
        \leqslant \frac{\sum\limits_{k=1}^{n}\D M_{m}}{n^2 g\lr*{\frac{1}{{m}}}^{2}t^2}
        \leqslant \frac{nC_2}{n^{2}t^2}
        = \frac{C_2}{nt^2} \rightarrow 0, \quad n \to \infty.
    \]
    Hence
    \[
        \frac{\sum\limits_{k=1}^{n}M_{m}^{(k)} - \sum\limits_{k=1}^{n}\E M_{m}}{ n g\lr*{\frac{1}{{m}}}} \xrightarrow{\Prob} 0, \quad n \to \infty.
    \]
    By the asymptotics for expectations (Lemma~\ref{lm:max-expect}),
    \begin{equation}
        \label{eq:upper-bound}
        \frac{\sum\limits_{k=1}^{n}M_{m}^{(k)}}{ n g\lr*{\frac{1}{{m}}}}  \xrightarrow{\Prob} 1, \quad n \to \infty.
    \end{equation}

    For the lower bound, let $N$ be such that \eqref{eq:var-bound} holds for all $n\ge N$.
    Apply Chebyshev to
    $\frac{\sum\limits_{k=m-n+N}^{m} M_{k}^{(k)}}{n\, g\!\lr*{\tfrac{1}{m}}}$. Using \eqref{eq:var-bound} and the monotonicity of $g(\cdot)$, for any $t$,

    \begin{eqnarray*}
        &&\Prob\set*{\abs*{
            \frac{\sum\limits_{k=m - n + N}^{m}M_k^{(k)}}{n g\lr*{\frac{1}{{m}}}} -
            \frac{\sum\limits_{k=m - n + N}^{m} \E M_{k}}{n g\lr*{\frac{1}{{m}}}}
        } \geqslant t}
        \leqslant \frac{\sum\limits_{k=m - n + N}^{m}\D M_{k}}{n^2 g\lr*{\frac{1}{{m}}}^{2}t^2}\\
        &\leqslant& \frac{1}{n^{2}t^2}\sum\limits_{k=m - n + N}^{m}\frac{\D M_{k}}{g\lr*{\frac{1}{{k}}}^{2}}
        \leqslant \frac{nC_2}{n^{2}t^2} = \frac{C_2}{nt^2}\rightarrow 0, \quad m \to \infty.
    \end{eqnarray*}

    Therefore,
    \[
        \frac{\sum\limits_{k=m - n + N}^{m}M_k^{(k)} - \sum\limits_{k=m - n + N}^{m}\E M_{k}}{n g\lr*{\frac{1}{{m}}}} \xrightarrow{\Prob} 0, \quad n \to \infty.
    \]

    Note that for fixed $N$,
    \[
        \frac{\sum\limits_{k=m - n + 1}^{m - n + N - 1}M_k^{(k)} - \sum\limits_{k=m - n + 1}^{m - n + N - 1} \E M_{k} }{n g\lr*{\frac{1}{{m}}}} \xrightarrow{\Prob} 0, \quad n \to \infty.
    \]
    Summing up these terms yields
    \begin{equation}
        \frac{\sum\limits_{k=m - n + 1}^{m}M_k^{(k)} - \sum\limits_{k=m - n + 1}^{m}\E M_{k}}{n g\lr*{\frac{1}{{m}}}} \xrightarrow{\Prob} 0, \quad n \to \infty.
        \label{eq:centerd-prob-lim}
    \end{equation}

    From \eqref{eq:sum-expect-lower-bound} we have
    \begin{equation}
        \liminf\limits_{n \to \infty}\frac{\sum\limits_{k=m - n + 1}^{m}\E M_{k}}{n g\lr*{\frac{1}{{m}}}} \geqslant 1.
        \label{eq:liminf-expect-sum}
    \end{equation}
    For the upper limit, using the monotonicity of $g(\cdot)$,

    \[
        \sum\limits_{k=m - n + 1}^{m}\E M_{k} =
        \sum\limits_{k=m - n + 1}^{m} g\lr*{\frac{1}{k}}\cdot (1 + o(1)) \leqslant
        n \cdot g\lr*{\frac{1}{m}}\cdot (1 + o(1)).
    \]

    Hence
    \begin{equation}
        \limsup\limits_{n \to \infty} \frac{\sum\limits_{k=m - n + 1}^{m}\E M_{k}}{n g\lr*{\frac{1}{{m}}}} \leqslant 1.
        \label{eq:limsup-expect-sum}
    \end{equation}
    Combining \eqref{eq:liminf-expect-sum} and \eqref{eq:limsup-expect-sum} we get
    \begin{equation*}
        \frac{\sum\limits_{k=m - n + 1}^{m}\E M_{k}}{n g\lr*{\frac{1}{{m}}}} \to 1,\quad n \to \infty.
    \end{equation*}
    Putting this together with \eqref{eq:centerd-prob-lim}, we obtain
    
    \begin{equation}
        \label{eq:lower-bound}
        \frac{\sum\limits_{k=m - n + 1}^{m}M_k^{(k)}}{n g\lr*{\frac{1}{{m}}}} \xrightarrow{\Prob} 1, \quad n \to \infty.
    \end{equation}

    Finally, using the two-sided bound \eqref{eq:two-way-bound} together with \eqref{eq:upper-bound} and \eqref{eq:lower-bound} and the squeeze argument for convergence in probability, we conclude that
    \[\frac{\M_{n,m}}{ n g\lr*{\frac{1}{{m}}}} \xrightarrow{\Prob}  1, \quad n \to \infty. \qedhere\]
\end{proof}

\section{Exponential tail bounds}\label{sec:exp}
After establishing a law of large numbers for the random assignment process, a natural question arises: does an analogue of the large deviations principle hold for it?
As in the previous sections, we will work with upper and lower bounds for the random variable $\M_{n, m}$.
Below we present four exponential bounds: upper and lower bounds for the left and right tails of the distribution of $\M_{n, m}$.

We will use the following notation:
\begin{enumerate}
    \item $\overline{F}_X(t) = \Prob\set*{X \geqslant t}$,
    \item $F_X(t) = \Prob\set*{X \leqslant t}$,
    \item $\Gamma(\cdot)$ is Euler’s gamma function,
\end{enumerate}

A random variable $G$ has the \textit{Gumbel distribution} if its distribution function is $F_G(r) = e^{-e^{-r}}$.
This distribution has the following properties:
\begin{enumerate}
    \item $\E G = \gamma$, the Euler--Mascheroni constant,
    \item max-stability:
    \[
        \max_{1 \leqslant k \leqslant n} G^{(k)} \sim G + \log n.
    \]
    \item $\Lambda_{gum}(t) = \log \E e^{t(G - \gamma)} = \log \Gamma(1-t) - \gamma t$ is the cumulant generating function (log-mgf) of the centered Gumbel distribution,
    \item $\Lambda_{gum}^*(r) \coloneqq \sup\limits_{t \in \R}\lr*{tr - \Lambda_{gum}(t)}$ is its large-deviation rate function.
\end{enumerate}

A random variable $\Phi_\alpha$ has the \textit{Fréchet distribution} with parameter $\alpha$ if its distribution function is
\[
    F_{\Phi_\alpha}(r) =
\begin{cases}
    0, & r \leqslant 0, \\
    \exp\left(-r^{-\alpha}\right), & r > 0.
\end{cases}
\]
This distribution has the following properties:
\begin{enumerate}
    \item max-stability:
    \[
        \max_{1 \leqslant k \leqslant n} \Phi_\alpha^{(k)} \sim n^{\frac{1}{\alpha}}\Phi_\alpha.
    \]
    \item If $0 < k < \alpha$, then $\E \lr*{\Phi_\alpha}^k = \Gamma\!\left(1-\frac{k}{\alpha}\right)$.
\end{enumerate}

A random variable $\Psi_\alpha$ has the \textit{Weibull distribution} with parameter $\alpha > 0$ if its distribution function is
\[
F_{\Psi_\alpha}(r) =
\begin{cases}
\exp\left(-(-r)^\alpha\right), & r \leqslant 0, \\
1, & r > 0.
\end{cases}
\]
This distribution has the following properties:
\begin{enumerate}
    \item max-stability:
    \[
        \max_{1 \leqslant k \leqslant n} \Psi_\alpha^{(k)} \sim n^{-\frac{1}{\alpha}} \Psi_\alpha,
    \]
    \item For $k > 0$,
    \[
        \E \left(-\Psi_\alpha\right)^k = \Gamma\!\left(1 + \frac{k}{\alpha}\right).
    \]
\end{enumerate}

\begin{lemma}[Fisher--Tippett--Gnedenko theorem \cite{FT28, G43}]
    \label{lm:th-fisher-tippet-gnedenko}
    Let $\set*{X_{k}}_{k = 1}^{\infty}$ be a sequence of i.i.d.\ random variables. Define
    \[
    h(t) = \frac{\int_t^{\infty} \overline{F}_X(x) \, dx}{\overline{F}_X(t)}.
    \]
    Let $g(p)$ denote the upper quantile function of the distribution of $X$.
    Assume that for all $u > 0$,
    \[
        \frac{g\lr*{\frac{1}{ut}} - g\lr*{\frac{1}{t}}}{g\lr*{\frac{1}{et}} - g\lr*{\frac{1}{t}}} \to \log{u}, \quad t \to \infty.
    \]
    Then $\frac{M_n - b_n}{a_n} \Rightarrow G,$ where $b_n = g\!\lr*{\frac{1}{n}}$ and $a_n = h(b_n)$.
\end{lemma}

A random variable $X$ is said to have a \textit{right tail of exponential type} with parameter $c$ if $X$ satisfies the condition of Lemma~\ref{lm:th-fisher-tippet-gnedenko} and $\lim\limits_{n \to \infty} a_n = c$.

\begin{remark}
    The class of variables with a right tail of exponential type includes the Gumbel distribution and the exponential distribution.
\end{remark}
\begin{remark}
    It is easy to show that if $X$ has a right tail of exponential type with parameter $c$, then the quantile function $g(p)$ is slowly varying at $0$. 
    The converse statement is false.
\end{remark}

\begin{theorem}
    \label{th:exponential-bound-exp-tail}
    Consider a sequence of random matrices $(X_{ij})$ of size $n \times m$ whose entries are i.i.d.\ random variables, where $m = m(n)$ and $m \geqslant n$.
    Suppose that
    \begin{enumerate}
        \item $\E X^{-} < \infty$,
        \item $X$ has a right tail of exponential type with parameter $c$.
    \end{enumerate}
    Define $\varepsilon\lr*{m,n} \coloneq g\!\lr*{\frac{1}{m}} - \frac{1}{n}\sum\limits_{k = 1}^{n} g\!\lr*{\frac{1}{m - n + k}}$, $\varepsilon_1 = \liminf\limits_{n, m \to \infty} \varepsilon\lr*{m,n}$, and $\varepsilon_2 = \limsup\limits_{n, m \to \infty} \varepsilon\lr*{m,n}$.

    Then the following bounds hold:
    \begin{eqnarray*}
        &&\limsup_{n,m \to \infty} \frac{1}{n}\log \Prob \set*{\M_{n, m} \geqslant n\lr*{cr + c\gamma + g\!\lr*{\frac{1}{m}}}} \leqslant -\Lambda_{gum}^{*}\!\lr*{r}  \quad r > 0; \\
        &&\liminf_{n,m \to \infty} \frac{1}{n}\log \Prob \set*{\M_{n, m} \geqslant n\lr*{cr + c\gamma + g\!\lr*{\frac{1}{m}}}} \geqslant -\Lambda_{gum}^{*}\!\lr*{r + \varepsilon_1} \quad r > -\varepsilon_1; \\
        &&\limsup_{n,m \to \infty} \frac{1}{n}\log \Prob \set*{\M_{n, m} \leqslant n\lr*{cr + c\gamma + g\!\lr*{\frac{1}{m}}}} \leqslant -\Lambda_{gum}^{*}\!\lr*{r + \varepsilon_2} \quad r < -\varepsilon_2; \\
        &&\liminf_{n,m \to \infty} \frac{1}{n}\log \Prob \set*{\M_{n, m} \leqslant n\lr*{cr + c\gamma + g\!\lr*{\frac{1}{m}}}} \geqslant -\Lambda_{gum}^{*}\!\lr*{r} \quad r < 0.
    \end{eqnarray*}
\end{theorem}

To prove the theorem we will need several auxiliary results.

\begin{lemma}[Gärtner--Ellis theorem \cite{G77, E84}]
    \label{th:gartner-ellis}
    Let $Z_k$ be a sequence of centered random variables.
    Suppose the limit $\Lambda_Z(t) \coloneq \lim\limits_{k \to \infty} \frac{1}{k} \log \E e^{t kZ_k}$ exists for all $t$ (possibly equal to $+\infty$).
    Let $\Lambda_Z^*(r) = \sup\limits_{t \in \R}\lr*{rt - \Lambda_Z(r)}$.

    Assume that $\Lambda_Z(t) < \infty$ in a neighborhood of zero.
    Then
    \begin{gather*}
        \liminf_{k \to \infty} \frac{1}{k} \log \Prob\set{Z_k \geqslant r} \geqslant -\Lambda_Z^*(r), \quad r > 0;\\
        \limsup_{k \to \infty} \frac{1}{k} \log \Prob\set{Z_k \leqslant r} \leqslant -\Lambda_Z^*(r), \quad r < 0.
    \end{gather*}
\end{lemma}

\begin{lemma}[Proposition 2.1 (i), Resnick (1987) \cite{R87}]
    \label{lm:proposition-resnick-freshet}
    Let $\set*{X_{i}}_{i = 1}^{\infty}$ be a sequence of i.i.d.\ random variables.
    Suppose there exists $0 < k < \alpha$ such that $\E \lr*{X^{-}}^k < \infty$, and there are $a_n$ such that
    \[\frac{M_n}{a_n} \Rightarrow \Phi_\alpha, \quad n \to \infty.\]
    Then $\E \lr*{\frac{M_n}{a_n}}^k \to \E \Phi_\alpha^k = \Gamma\lr*{1 - \frac{k}{\alpha}}, \quad n \to \infty$.
\end{lemma}

\begin{lemma}[Proposition 2.1 (ii), Resnick (1987) \cite{R87}]
    \label{lm:proposition-resnick-weibull}
    Let $\set*{X_{i}}_{i = 1}^{\infty}$ be a sequence of i.i.d.\ random variables whose distribution $F$ has a finite right endpoint $x_0$.
    Suppose there exists $0 < k < \alpha$ such that $\E \abs*{X}^k < \infty$, and there is a sequence $a_n$ for which
    \[
        \frac{x_0 - M_n}{a_n} \Rightarrow -\Psi_\alpha, \quad n \to \infty.
    \]
    Then
    \[
        \E \lr*{\frac{x_0 - M_n}{a_n}}^k \to \E \lr*{-\Psi_\alpha}^k =  \Gamma\lr*{1 + \frac{k}{\alpha}}, \quad n \to \infty.
    \]
\end{lemma}

\begin{lemma}[Proposition 2.1 (iii), Resnick (1987) \cite{R87}]
    \label{lm:proposition-resnick-gum}
    Let $\set*{X_{i}}_{i = 1}^{\infty}$ be a sequence of i.i.d.\ random variables.
    Suppose there exists $k > 0$ such that $\E \lr*{X^{-}}^k < \infty$, and there exist $a_n, b_n$ such that
    \[\frac{M_n - b_n}{a_n} \Rightarrow G, \quad n \to \infty.\]
    Then $\E \lr*{\frac{M_n - b_n}{a_n}}^k \to \E G^k, \quad n \to \infty$.
\end{lemma}

\begin{statement}[a corollary of Lemma \ref{lm:proposition-resnick-gum}]
    \label{st:expext-limit}
    Let $\set*{X_{i}}_{i = 1}^{\infty}$ be a sequence of i.i.d.\ random variables with a right tail of exponential type with parameter $c$.
    Suppose that $\E X^{-} < \infty$.
    Then $\E \frac{M_n - g\lr*{\frac{1}{n}}}{c} \to \E G = \gamma, \quad n \to \infty$.
\end{statement}

\begin{statement}
    \label{st:kumul-conv}
    Let $\set*{X_{i}}_{i = 1}^{\infty}$ be a sequence of i.i.d.\ random variables with a right tail of exponential type with parameter $1$.
    Then for all sufficiently large $n$ and any $t \in (-1, 1)$, the quantity $\E e^{t\lr*{M_n - \E M_n}}$ is well defined and
    \[
        \log \E e^{t\lr*{M_n - \E M_n}} \to \Lambda_{gum}(t), \quad n \to \infty, -1 < t < 1.
    \]
\end{statement}

\begin{proof}
    Note that $M_n - g\lr*{\frac{1}{n}} \Rightarrow G$ by Lemma~\ref{lm:th-fisher-tippet-gnedenko}.
    Apply the continuous map $\exp\set*{\cdot}$ to both sides to obtain
    \[
        e^{M_n - g\lr*{\frac{1}{n}}} \Rightarrow e^{G}, \quad n \to \infty.
    \]
    Using the fact that $e^{G} \sim \Phi_1$,
    set $Y_n = e^{X_n}$ and $A_n = e^{g\lr*{\frac{1}{n}}}$; then
    \[
        \frac{\max\set*{Y_1, \dots, Y_n}}{A_n} \Rightarrow \Phi_1, \quad n \to \infty.
    \]
    By Lemma~\ref{lm:proposition-resnick-freshet}, for $0 < t < 1$ we get
    \[
        \E e^{t\lr*{M_n - g\lr*{\frac{1}{n}}}} = \E \lr*{\frac{\max\set*{Y_1, \dots, Y_n}}{A_n}}^t \to \E \Phi_1^t = \Gamma\lr*{1 - t} \quad n \to \infty.
    \]
    Next, apply the continuous map $-\exp\set*{-\cdot}$ to both sides of the weak convergence to obtain
    \[
        -e^{-M_n + g\lr*{\frac{1}{n}}} \Rightarrow -e^{-G}, \quad n \to \infty.
    \]
    Using the fact that $-e^{-G} \sim -\Psi_1$,
    set $Z_n = -e^{-X_n}$, $B_n = -e^{-g\lr*{\frac{1}{n}}}$, and $x_0 = 0$. Then
    \[
        \frac{x_0 - \max\set*{Z_1, \dots, Z_n}}{B_n} \Rightarrow -\Psi_1, \quad n \to \infty.
    \]
    By Lemma~\ref{lm:proposition-resnick-weibull}, for $-1 < t < 0$ we obtain
    \[
        \E e^{t\lr*{M_n - g\lr*{\frac{1}{n}}}} = \E \lr*{\frac{-\max\set*{Z_1, \dots, Z_n}}{B_n}}^{-t} \to \E \lr*{-\Psi_1}^{-t} = \Gamma\lr*{1 - t} \quad n \to \infty.
    \]
    
    Note that $\E M_n - g\lr*{\frac{1}{n}} \to \gamma$ as $n \to \infty$ (Statement~\ref{st:expext-limit}); therefore
    \[
        \log \E e^{t\lr*{M_n - \E M_n}} \to \log \Gamma\lr*{1 - t} - t\gamma = \Lambda_{gum}(t), \quad n \to \infty, -1 < t < 1.
    \]
\end{proof}

\begin{proof}[Proof of Theorem \ref{th:exponential-bound-exp-tail}]
    Note that if a random variable $X$ has a right tail of exponential type with parameter $c$, then the random variable $\frac{X}{c}$ has a right tail of exponential type with parameter $1$.
    Hence we may assume $c = 1$.

    From Statement~\ref{st:expext-limit} it follows that
    \[
        \E M_k - g\lr*{\frac{1}{k}} - \gamma \to 0, \quad k \to \infty.
    \]

    Using the convergence of the cumulant of the centered maximum from Statement~\ref{st:kumul-conv}, we estimate the limiting cumulants of the centered upper and lower bounds for $\M_{n, m}$ from~\eqref{eq:two-way-bound}.

    Limiting cumulant of the upper bound:
    \begin{eqnarray*}
        &&
        \lim_{n,m \to \infty} \frac{1}{n}\log \E e^{t\lr*{\sum\limits_{k = m - n + 1}^{m} \lr*{M_m^{(k)} - \E M_m}}}
        \\
        &=&
        \lim_{n,m \to \infty} \log \E e^{t\lr*{M_m - \E M_m}}
        \\
        &=&
        \Lambda_{gum}(t).
        \\
    \end{eqnarray*}

    Limiting cumulant of the lower bound:
    \begin{eqnarray*}
        &&
        \lim_{n,m \to \infty} \frac{1}{n}\log \E e^{t\lr*{\sum\limits_{k = m - n + 1}^{m} \lr*{M_k - \E M_k}}}
        \\
        &=&
        \lim_{n,m \to \infty} \frac{\sum\limits_{k = m - n + 1}^{m} \log \E e^{t\lr*{M_m - \E M_m}}}{n}
        \\
        &=&
        \Lambda_{gum}(t).
        \\
    \end{eqnarray*}

    Apply Lemma~\ref{th:gartner-ellis} to the sequences $\frac{\sum\limits_{k = m - n + 1}^m \lr*{M_m^{(k)} - \E M_m}}{n}$ and $\frac{\sum\limits_{k = m - n + 1}^m \lr*{M_k - \E M_k}}{n}$ and write down four upper and lower bounds for the left and right tails.

    Lower bound for the left tail:
    \begin{eqnarray*}
        &&
        \liminf_{n, m \to \infty} \frac{1}{n} \log \Prob \set*{\M_{n, m} \geqslant n\lr*{r + \gamma + g\lr*{\frac{1}{m}}}}
        \\
        &\geqslant&
        \liminf_{n, m \to \infty} \frac{1}{n}\log \Prob \set*{\sum\limits_{k = m - n + 1}^m M_m^{(k)} \geqslant n\lr*{r + \gamma + g\lr*{\frac{1}{m}}}}
        \\
        &=&
        \liminf_{n, m \to \infty} \frac{1}{n}\log \Prob \set*{\frac{\sum\limits_{k = m - n + 1}^m \lr*{M_m^{(k)} - \E M_m}}{n} \geqslant r + o(1)}
        \\
        &=& -\Lambda_{gum}^{*}(r), \quad r \geqslant 0.
    \end{eqnarray*}

    Upper bound for the left tail:
    \begin{eqnarray*}
        &&
        \limsup_{n, m \to \infty} \frac{1}{n} \log \Prob \set*{\M_{n, m} \geqslant n\lr*{r + \gamma + g\lr*{\frac{1}{m}}}}
        \\
        &\leqslant&
        \limsup_{n, m \to \infty} \frac{1}{n}\log \Prob \set*{\sum\limits_{k = m - n + 1}^m M_k \geqslant n\lr*{r + \gamma + g\lr*{\frac{1}{m}}}}
        \\
        &=&
        \limsup_{n, m \to \infty} \frac{1}{n}\log \Prob \set*{\frac{\sum\limits_{k = m - n + 1}^m \lr*{M_k - \E M_k}}{n} \geqslant r + \varepsilon(n,m) + o(1)}
        \\
        &\leqslant& -\Lambda_{gum}^{*}(r + \varepsilon_2), \quad r \geqslant -\varepsilon_2.
    \end{eqnarray*}

    Upper bound for the right tail:
    \begin{eqnarray*}
        &&
        \limsup_{n, m \to \infty} \frac{1}{n} \log \Prob \set*{\M_{n, m} \leqslant n\lr*{r + \gamma + g\lr*{\frac{1}{m}}}}
        \\
        &\leqslant&
        \limsup_{n, m \to \infty} \frac{1}{n} \log \Prob \set*{\sum\limits_{k = m - n + 1}^m M_m^{(k)}  \leqslant n\lr*{r + \gamma + g\lr*{\frac{1}{m}}}}
        \\
        &=&
        \limsup_{n, m \to \infty} \frac{1}{n} \log \Prob \set*{\frac{\sum\limits_{k = m - n + 1}^m \lr*{M_m^{(k)} - \E M_m}}{n} \leqslant r + o(1)}
        \\
        &=& -\Lambda_{gum}^{*}(r), \quad r \leqslant 0.
    \end{eqnarray*}

    Lower bound for the right tail:
    \begin{eqnarray*}
        &&
        \liminf_{n, m \to \infty} \frac{1}{n} \log \Prob \set*{\M_{n, m} \leqslant n\lr*{r + \gamma + g\lr*{\frac{1}{m}}}}
        \\
        &\geqslant&
        \liminf_{n, m \to \infty} \frac{1}{n}\log \Prob \set*{\sum\limits_{k = m - n + 1}^m M_k \leqslant n\lr*{r + \gamma + g\lr*{\frac{1}{m}}}}
        \\
        &=&
        \liminf_{n, m \to \infty}  \frac{1}{n}\log \Prob \set*{\frac{\sum\limits_{k = m - n + 1}^m \lr*{M_k - \E M_k}}{n} \leqslant r + \varepsilon(n,m)  + o(1)}
        \\
        &\geqslant& -\Lambda_{gum}^{*}(r + \varepsilon_1), \quad r \leqslant -\varepsilon_1.
    \end{eqnarray*}
\end{proof}

\section{Conclusion}\label{sec:concl}
% !TEX root = main_en.tex
% English version of: 5 - Заключение.tex

The author is grateful to M.\,A.~Lifshits for posing the problems and for helpful advice.

\section*{References}\label{sec:bib}
% !TEX root = main_en.tex
% English alias of: 6 - Список литературы.tex
\patchcmd{\thebibliography}{\section*{\refname}}{}{}{}

\end{document}